%% file: main.tex
\documentclass[11pt,a4paper]{amsart}

\usepackage[
  left=25mm,
  right=25mm,
  top=20mm,
   bottom=20mm
]{geometry} %

\usepackage{microtype}
\usepackage{etoolbox}
\usepackage{blkarray}
\usepackage{booktabs}

\usepackage{url,amssymb,enumerate,colonequals}

\RequirePackage{iftex}
\ifLuaTeX
  \RequirePackage{luatex85}
  \ifdefined\pdftexversion\else\chardef\pdftexversion=140\fi
\fi

\usepackage[all]{xy}
\usepackage[all]{xy} %
\usepackage{mathrsfs} %
\usepackage[dvipsnames]{xcolor}
\usepackage{tikz-cd}
\usepackage[normalem]{ulem} %
\usepackage{textcomp}
\usepackage{relsize}   %

\definecolor{linkblue}{RGB}{20,60,130}
\usepackage[
pdfauthor={Hidde Lammert Bakker and Emre Can Sertöz},
pdftitle={A generalised cross-ratio and limits of local heights},
colorlinks=true,
allcolors=linkblue
]{hyperref}

\makeatletter
\renewcommand\subsection{\@startsection{subsection}{2}{\z@}
                                   {-1.5ex \@plus -.5ex \@minus -.2ex}%
                                   {1.0ex \@plus.2ex}%
                                   {\normalfont\itshape}}%
\makeatother
\makeatletter
\g@addto@macro\bfseries{\boldmath}
\makeatother

\newcounter{nodenum}[section] 
\renewcommand{\thenodenum}{\thesubsection.\arabic{nodenum}} 
\newcommand{\node}[1]{\par\refstepcounter{nodenum}\vspace{1ex}\noindent(\thenodenum)\space \textbf{#1}}

\makeatletter

\@addtoreset{equation}{section} %
\makeatother

\usepackage[style=ieee-alphabetic,backend=biber, isbn=false, opcittracker=true,
  url=true
]{biblatex}
\AtEveryBibitem{%
  \clearfield{month}%
}

\usepackage{color}
\usepackage{marvosym} 

\usepackage[OT2,T1]{fontenc}
\DeclareSymbolFont{cyrletters}{OT2}{wncyr}{m}{n}
\DeclareMathSymbol{\Sha}{\mathalpha}{cyrletters}{"58}

\newcommand{\C}{\mathbb{C}}

\newcommand{\G}{\mathbb{G}}
\renewcommand{\H}{\operatorname{H}}

\newcommand{\Pp}{\mathbb{P}}
\newcommand{\p}{\mathbb{P}^1}

\newcommand{\pn}{\mathbb{P}^{n}}
\newcommand{\pv}{\mathbb{P}(V)}
\newcommand{\Q}{\mathbb{Q}}

\newcommand{\R}{\mathbb{R}}

\newcommand{\Z}{\mathbb{Z}}
\newcommand{\Zloc}{\mathbf{Z}} 

\newcommand{\xc}{\mathfrak{c}}

\newcommand{\xm}{\mathfrak{m}}

\makeatletter
\newcommand\RedeclareMathOperator{%
  \@ifstar{\def\rmo@s{m}\rmo@redeclare}{\def\rmo@s{o}\rmo@redeclare}%
}
\newcommand\rmo@redeclare[2]{%
  \begingroup \escapechar\m@ne\xdef\@gtempa{{\string#1}}\endgroup
  \expandafter\@ifundefined\@gtempa
     {\@latex@error{\noexpand#1undefined}\@ehc}%
     \relax
  \expandafter\rmo@declmathop\rmo@s{#1}{#2}}
\newcommand\rmo@declmathop[3]{%
  \DeclareRobustCommand{#2}{\qopname\newmcodes@#1{#3}}%
}
\@onlypreamble\RedeclareMathOperator
\makeatother

\DeclareMathOperator{\Char}{char}

\DeclareMathOperator{\Ext}{Ext}

\DeclareMathOperator{\Gr}{Gr}

\RedeclareMathOperator{\hom}{hom}

\DeclareMathOperator{\proj}{Proj}

\DeclareMathOperator{\rank}{rank}

\DeclareMathOperator{\Spec}{Spec}

\DeclareMathOperator{\Sym}{Sym}

\DeclareMathOperator{\diag}{diag}
\DeclareMathOperator{\reg}{reg}

\newcommand{\sat}{{\operatorname{sat}}}

\renewcommand{\span}{\operatorname{span}}

\newcommand{\dd}{\mathrm{d}}

\usepackage{mathtools}

\theoremstyle{plain}
\newtheorem{theorem}{Theorem}[section]
\newtheorem{proposition}[theorem]{Proposition}
\newtheorem{lemma}[theorem]{Lemma}

\newtheorem{introtheorem}{Theorem}       

\newtheorem{corBprime}{Corollary}          
\newtheorem{corCprime}{Corollary}          
\newtheorem{corCsecond}{Corollary}         

\theoremstyle{definition}
\newtheorem{definition}[theorem]{Definition}

\theoremstyle{remark}
\newtheorem{remark}[theorem]{Remark}

\newcommand{\abs}[1]{\lvert #1 \rvert}

\newcommand{\set}[1]{\{ #1 \}}

\newcommand{\tpip}{(2\pi{\tt i})}

\DeclareMathOperator{\id}{id}

\DeclareMathOperator{\MHS}{MHS}
\newcommand{\ZMHS}{\Z\text{-}\MHS}
\newcommand{\RMHS}{\R\text{-}\MHS}

\newcommand{\ca}{\mathcal{A}}
\newcommand{\cb}{\mathcal{B}}

\newcommand{\ch}{\mathcal{H}}

\newcommand{\ci}{\mathcal{I}}

\newcommand{\cl}{\mathcal{L}}

\newcommand{\co}{\mathcal{O}}
\newcommand{\cq}{\mathcal{Q}}

\newcommand{\cs}{\mathcal{S}}

\newcommand{\cw}{\mathcal{W}}
\newcommand{\cv}{\mathcal{V}}

\newcommand{\too}{\longrightarrow}
\newcommand{\tos}{\twoheadrightarrow}
\newcommand{\toi}{\hookrightarrow}

\newcommand{\coleq}{\colonequals}

\newcommand{\xci}[1][]{\xc_{\mathrm{lim}#1}} 
\newcommand{\xca}[1][]{\xc_{\mathrm{ar}#1}} 

\newcounter{arrowcounter}

\newcommand{\circled}[1]{%
  \tikz[baseline=(char.base)]{%
    \node[draw, circle, inner sep=1pt] (char){\ifmmode #1\else\scriptsize #1\fi};%
  }%
}

\newcommand{\circledref}[1]{%
  {\begingroup
     \hypersetup{hidelinks}%
     \hyperref[#1]{\circled{\ref*{#1}}}%
   \endgroup}%
}

\author[H.~L.~Bakker]{Hidde Lammert Bakker}
\address{Hidde Lammert Bakker, Faculty of Electrical Engineering, Mathematics and Computer Science, Delft University of Technology, Mekelweg 4, 2628 CD Delft, The Netherlands}
\email{math@hlbakker.nl}

\author[E.~C.~Sertöz]{Emre Can Sertöz}
\address{Emre Can Sertöz, Mathematical Institute, Leiden University, Einsteinweg 55, 2333 CC Leiden, The Netherlands}
\email{emre@sertoz.com}
\urladdr{\url{https://emresertoz.com}}

\title[Generalised cross-ratio]{A generalised cross-ratio and limits of local heights}
\subjclass[2020]{14G40, 14C17, 32G20 (Primary); 14C30, 14N20, 14C25 (Secondary)}

\begin{document}

\begin{abstract}
  \input{sections/abstract}
\end{abstract}

\maketitle

\input{sections/00_introduction}
\input{sections/01_correspondence}
\input{sections/02_intermezzo}
\input{sections/03_intersection}
\input{sections/04_limit}

\printbibliography

\end{document}

%% file: sections/abstract.tex
We generalise the standard cross-ratio of four points on a projective line to a cross-ratio of a configuration of four planes in projective $n$-space, the first pair $A_1,\,A_2$ being $k$-dimensional and the second pair $B_1,\, B_2$ being $(n-k-1)$-dimensional, with $A_i \cap B_j = \emptyset$. Over the complex numbers, we show that this cross-ratio equals the augmented height pairing of the corresponding cycles $A_1-A_2, \, B_1-B_2$. Over a discretely valued field, we show that the valuation of the cross-ratio equals the intersection degree of the cycles once they are spread out over the valuation ring. Putting the two together, we conclude that the asymptotics of the Archimedean height pairing of a holomorphic family of configurations are governed by this intersection degree. We also define a degenerate cross-ratio for when $A_i \cap B_j \neq \emptyset$ and interpret the ``limit height'' of a degenerating holomorphic family of planes as the degenerate cross-ratio of the central plane configuration.

%% file: sections/00_introduction.tex

\section{Introduction}

Let $K$ be a field. The ``lowly'' cross-ratio \cite[p.651]{Vakil2003} is a function of four distinct points $p_1,\dots,p_4 \in \p(K)$ on the projective line and outputs their only projective invariant. 
If, by change of coordinates, the first three points are brought into the position $p_1=[1:0]$, $p_2=[0:1]$, and $p_3=[1:1]$, then the fourth point will be $p_4 = [\lambda:1]$, and we can define the cross-ratio to be the coordinate $\lambda \in K$. 

The value of $\lambda$ can also be expressed as follows: 
If $p_i$ represents the line $K \cdot v_i \subset K^2$ and $\ell_i$ is a linear form vanishing on $p_i$ then the cross-ratio evaluates to
\begin{equation}\label{eq:classical}
 \frac{\ell_3(v_1)}{\ell_3(v_2)}\frac{\ell_4(v_2)}{\ell_4(v_1)} \in K^\times.
\end{equation}
This formula can be generalised in an obvious manner, as we now explain.

Take an $(n+1)$-dimensional $K$-vector space $V$. A linear subspace $A \subset \Pp(V) \simeq \pn$ of dimension $k$ is the projectivisation of a $(k+1)$-dimensional subspace $\hat A \subset V$. Note that the annihilator $\hat A^\perp$ of $\hat A$ in the dual space $V^\vee$ is of dimension $(n-k)$, which then defines a linear subspace $A^\perp \subset \Pp(V^\vee)$ of dimension $(n-k-1)$. 

For a finite-dimensional vector space $W$ we write $\det W \coleq \bigwedge^{\dim W} W$ for its top exterior power. Let $A, \, B \subset \pv$ be linear subspaces of dimensions $k$ and $(n-k-1)$ and choose generators $a \in \det \hat A$ and $b^\perp \in \det \hat B^\perp$. The duality between $\bigwedge^{k+1} V$ and $\bigwedge^{k+1} V^\vee$ allows us to pair the two generators $b^\perp(a) \in K$. Note that the elements $a$ and $b^\perp$ depend on the choice of generators, but only up to a scalar.

Concretely, choosing bases $\hat A = K\cdot\langle v_0,\dots,v_k \rangle$ and $\hat B^\perp = K\cdot\langle \ell^0,\dots,\ell^k \rangle$ we have generators $a = v_0 \wedge \dots \wedge v_k$ and $b^\perp = \ell^0 \wedge \dots \wedge \ell^k$ and $b^\perp(a) = \det (\ell^i(v_j)) \in K$, which will be non-zero precisely when $A$ and $B$ are disjoint. 

A \emph{complementary configuration of type $(k,n)$} is a quadruple $A_1,\, A_2,\, B_1, \, B_2 \subset \pv \simeq \pn$ where the $A_i$ are linear of dimension $k$ and the $B_j$ are linear of dimension $(n-k-1)$ and each $A_i$ is disjoint from each $B_j$. 

\begin{definition}\label{definintiongencr}
  The cross-ratio of the complementary configuration $A_1,\, A_2,\, B_1, \, B_2$  is the number 
  \begin{equation}
    (A_1,A_2;B_1,B_2) \coleq \frac{b_1^\perp(a_1)}{b_1^\perp(a_2)}\frac{b_2^\perp(a_2)}{b_2^\perp(a_1)} \in K^\times.
  \end{equation}
\end{definition}

This expression is evidently independent of the choice of generators. Moreover, by construction, it is independent of the choice of coordinates on $\pv$. If $n=1,\ k=0$ then the definition above agrees with the classical formula~\eqref{eq:classical}. We give equivalent formulations and state all symmetries of the cross-ratio in Section~\ref{sec:intermezzo}, including the action of the full symmetry group $S_4$ when $n=2k+1$.

\medskip

Now assume that $K$ is a subfield of the complex numbers. Then, it is well-known that the cross-ratio equals the augmented height pairing of the cycles $p_1-p_2$ and $p_3-p_4$, namely the extension class $\H^1(\p -\{p_3,p_4\},\{p_1,p_2\})(1) \in \Ext^1_{\ZMHS}(\Z(0),\Z(1)) \simeq \C^\times$.
Our first result generalises this observation to higher dimensions.

Take a complementary configuration of type $(k,n)$, $A_i,\, B_j\subset \pv$.
The cycles $A_1-A_2$ and $B_1-B_2$ are homologous to zero in $\pv$ and their \emph{augmented height pairing} is defined to be the value of the extension class 
\begin{equation}\label{eq:hp}
    \ch_\Pp \coleq \H^{2k+1}(\pn  -  B_1\cup B_2 , A_1\cup A_2)(k+1) \in \Ext^1_{\ZMHS}(\Z(0),\Z(1)) \simeq \C^\times.
\end{equation}

The identification $\Ext^1_{\ZMHS}(\Z(0),\Z(1)) \simeq \C^\times$ is canonical~\cite{Carlson1980}; however, viewing $\ch_\Pp$ as an element of $\Ext^1_{\ZMHS}(\Z(0),\Z(1))$ requires making a choice regarding the integral generators of the two non-trivial weight graded pieces of $\ch_\Pp$, an \emph{orientation} as in~\cite[Definition~2.3]{BurgosGilGoswamiPearlstein2022}; this choice is implied by the cycles $A_1-A_2$ and $B_1-B_2$. 

\begin{introtheorem}
\label{introthm:aug_height_is_cross}
    The augmented height pairing of $A_1-A_2$ and $B_1-B_2$ equals the cross-ratio $(A_1,A_2;B_1,B_2)$.
\end{introtheorem}

Our remaining results, Theorems~\ref{introthm:dvr} and~\ref{introthm:reg}, pursue the identity of Theorem~\ref{introthm:aug_height_is_cross} in a variational setting: the leading asymptotics of the height pairing are governed by a valuation of the cross-ratio, and its limiting value is a \emph{degenerate} cross-ratio, which lends the limit a ``motivic'' interpretation.

\begin{remark}
Theorem~\ref{introthm:aug_height_is_cross} agrees with the main result of Goncharov's recent paper~\cite{Goncharov2026}, but with a different proof; we use a correspondence, whereas he uses induction. Both papers grew out of a conversation with Spencer Bloch at the Regulators V conference in Pisa, 2024. However, we were unaware of Goncharov's work until it appeared on arXiv, coincidentally the same day Bakker defended his thesis~\cite{Bakker2026} on which this article is based.
\end{remark}

  \medskip

 By Hain \cite[Prop.~3.3.7]{Hain1990}, the Archimedean height pairing of $A_1-A_2$ and $B_1-B_2$ equals the obstruction to splitting the underlying real MHS, that is, the image of the extension class of $\ch_\Pp$ under
\[
\Ext^1_{\ZMHS}(\Z(0),\Z(1))\simeq\C^\times \xrightarrow{\ \log|\cdot|\ } \R \simeq \Ext^1_{\RMHS}(\R(0),\R(1)).
\]
Together with Theorem~\ref{introthm:aug_height_is_cross}, this gives
\begin{equation}\label{eq:arch_height}
\langle A_1-A_2,B_1-B_2\rangle_\infty = \log|(A_1,A_2;B_1,B_2)|.
\end{equation}

Suppose $K=\Q$. Because $A_1-A_2$ is linearly equivalent to $0$, the sum of the local height pairings over all places must be $0$. Taken together with the formula~\eqref{eq:arch_height} for the Archimedean height, this implies that the local height at a prime $p$ is
\begin{equation}\label{eq:non_arch_height}
    \langle A_1-A_2,B_1-B_2\rangle_p = -v_p(A_1,A_2;B_1,B_2)\log p,
\end{equation}
as the logarithms of primes are linearly independent over $\Q$. 

Since the non-Archimedean height pairings can be computed by an intersection product over a local model, we conclude that the degeneration of a configuration modulo $p$ is measured by the $p$-adic valuation of the cross-ratio. The analogue of the formula~\eqref{eq:non_arch_height} holds over any number field; we deduce it from the more general fact below.

\medskip

Now let $R$ be a discrete valuation ring (DVR) with fraction field $K$, maximal ideal $\xm$, and valuation $v_R$. Let $\underline A_i,\underline B_j\subset\pn_R$ be the Zariski closures of the cycles $A_i,B_j \subset \pn_K$. Following Fulton~\cite[Ch.~20]{Fulton1998}, these cycles have a well-defined intersection product whose degree we denote by $\deg\big((\underline A_1-\underline A_2)\cdot(\underline B_1-\underline B_2)\big)\in\Z$.

Note that these cycles are disjoint over the generic fibre, and may only intersect over the special fibre. The expected dimension of the intersection locus is $0$ and the intersection product is a $0$-cycle supported over $\xm$; however, the true intersection may have excess dimension.

\begin{introtheorem}\label{introthm:dvr}
    The intersection degree equals the valuation of the cross-ratio: 
   \[ \deg\big((\underline A_1-\underline A_2)\cdot(\underline B_1-\underline B_2)\big) = v_R(A_1,A_2;B_1,B_2).\]
\end{introtheorem}

Consider the following application. Let $\Delta \subset \C$ be a complex disk centred at $0$ and $\Delta^* = \Delta - \set{0}$. Suppose $\underline A_i,\underline B_j \subset \pn_\Delta$ form a flat holomorphic family whose fibres $\underline A_i(t),\underline B_j(t)$ constitute a complementary configuration of type $(k,n)$ for every $t\in\Delta^*$. By the Archimedean formula \eqref{eq:arch_height}, applied to the fibre over $t$, we have
\[
\langle\underline A_1(t)-\underline A_2(t),\underline B_1(t)-\underline B_2(t) \rangle_\infty = \log|(\underline A_1(t),\underline A_2(t);\underline B_1(t),\underline B_2(t))|.
\]
Note that the expression on the right is the logarithm of a meromorphic function of $t \in \Delta$, and therefore admits the asymptotic expansion 
\begin{equation}\label{eq:arch_divergence}
  \langle\underline A_1(t)-\underline A_2(t),\underline B_1(t)-\underline B_2(t) \rangle_\infty = m \log|t| + \log |u(t)|, \quad |t| \ll 1
\end{equation}
where $m$ is an integer and $u(t)$ is holomorphic and non-zero at $t=0$.

\begin{corBprime}\label{introcor:asymptotics}
    The coefficient $m$ of $\log|t|$ in the asymptotic expansion of the height pairing is the intersection degree of the cycles over $t=0$,
    \[
    m = \deg\big((\underline A_1-\underline A_2)\cdot(\underline B_1-\underline B_2)\big).
    \]
\end{corBprime}
This corollary follows immediately from Theorem~\ref{introthm:dvr} by passing to $R=\co_{\Delta,0}$, the ring of germs of holomorphic functions at $0$, which is a DVR.  The coefficient $m$ is the order of vanishing of the cross-ratio at $t=0$, i.e., the $t$-adic valuation of the cross-ratio, which equals the intersection degree by the theorem.

More generally, Chen studied the rate of divergence of the Archimedean height pairing of properly degenerating algebraically trivial cycles and he predicts that divergence is governed by their intersection degree
~\cite[Conjecture~1.5]{Chen2025}. This has been proven in the case of degenerating curves by Holmes and De Jong~\cite{HolmesDeJong2015}. Corollary~\ref{introcor:asymptotics} proves this conjecture in the case of degenerating planes of any dimension. Moreover, since Theorem~\ref{introthm:dvr} requires no properness assumption, Corollary~\ref{introcor:asymptotics} suggests that the conjecture may hold for arbitrary degenerations. 

\medskip

We turn to giving a ``motivic'' interpretation of the limit period $u(0)$ by describing it in terms of a further generalisation of the cross-ratio. On the special fibre, the four planes may meet, and we need to introduce a cross-ratio of such intersecting tuples.

It will be convenient here to use an equivalent formulation (Section~\ref{sec:intermezzo}) of the cross-ratio for a complementary configuration
\begin{equation}\label{eq:cr-wedge}
  (A_1,A_2;B_1,B_2) = \frac{a_1 \wedge b_1}{a_1 \wedge b_2}\, \frac{a_2 \wedge b_2}{a_2 \wedge b_1}.
\end{equation}

A \emph{degenerate complementary configuration of type $(k,n)$} is a tuple $A_1,\, A_2,\, B_1,\, B_2 \subset \pv$ with $\dim A_i = k$, $\dim B_j = n-k-1$; no assumption is made regarding the intersections $A_i \cap B_j$. Hence $a_i \wedge b_j$ may be zero. For emphasis, we refer to the case when all four intersections $A_i \cap B_j$ are empty as \emph{non-degenerate}.

Let $\hat S_{ij} = \hat A_i \cap \hat B_j$, and $\hat Q_{ij} = V/(\hat A_i + \hat B_j)$; note that $\dim \hat S_{ij} = \dim \hat Q_{ij}$. When $\hat S_{ij} \neq 0$ we have a \emph{canonical} isomorphism, Section~\ref{sec:limit},
\begin{equation}\label{eq:detV}
  \det V \simeq \det \hat A_i \otimes \det \hat B_j \otimes \hom(\det \hat S_{ij}, \det \hat Q_{ij}).
\end{equation}

Let $\xc = (\xc_{11}, \xc_{12}, \xc_{21}, \xc_{22})$ with each $\xc_{ij} \in \hom(\det \hat S_{ij}, \det \hat Q_{ij})$ a non-zero element. When $\hat S_{ij} = 0$, and hence $\hat Q_{ij} = 0$, both determinants are the empty exterior power $\det 0 = K$ by convention and so $\xc_{ij} \in K^\times$.

\begin{definition}
  The \emph{degenerate cross-ratio} of a degenerate complementary configuration is
  \begin{equation}\label{eq:xc-cr}
    (A_1,A_2;B_1,B_2)_\xc \coleq
      \frac{a_1 \otimes b_1 \otimes \xc_{11}}{a_1 \otimes b_2 \otimes \xc_{12}}\,
      \frac{a_2 \otimes b_2 \otimes \xc_{22}}{a_2 \otimes b_1 \otimes \xc_{21}} \ \in K^\times,
  \end{equation}
where the ratios are computed in $\det V$ using the canonical identifications~\eqref{eq:detV} for each~$i,j$. 
\end{definition}

The degenerate cross-ratio is independent of the generators $a_i, b_j$, but it depends on $\xc$; more precisely, it depends only on the alternating tensor $\xc^\otimes \coleq \xc_{11} \otimes \xc_{22} \otimes \xc_{12}^{-1} \otimes \xc_{21}^{-1}$. When every $\hat S_{ij} = 0$, we may pick $\xc = (1,1,1,1)$ and recover~\eqref{eq:cr-wedge}, the cross-ratio of a complementary configuration.

\medskip

The important point is that, in the presence of extra structure, a natural $\xc$ is induced. We will consider two cases: First, when the degenerate complementary configuration is attained as the special fibre of a generically non-degenerate family, there is an induced ``limit perturbation'' $\xci$. Second, when the base field is a number field of class number $1$, we get an ``arithmetic perturbation'' $\xca$.

Let $R$ be a DVR, with $K,\, \xm$, and $v_R$ as above, $\kappa$ the residue field, and $t$ a uniformiser. 
We define the \emph{$t$-regularised limit} as the map
\begin{equation}\label{eq:reg-limit}
  \reg_t \colon K^\times \to \kappa^\times, \qquad  x \mapsto t^{-v_R(x)} x \mod \xm,
\end{equation}
which depends on $t$ only through $t \bmod \xm^2$. Our application of the map $\reg_t$ is in the spirit of the regularisation of period integrals of~\cite{DupontPanzerPym2026}, here in algebraic guise.

Let $\underline A_i, \underline B_j \subset \pn_R$ be the closures of a non-degenerate configuration $A_i, B_j$ over $K$, and denote the special fibres as $A_i^s, B_j^s \subset \pn_\kappa$ which form a degenerate configuration. The family induces a natural datum $\xci$, depending only on $t \bmod \xm^2$, recording the infinitesimal separation of the meeting planes, see Proposition~\ref{prop:limit-perturbation}. Geometrically, the way in which this degenerate configuration comes apart induces an ``infinitesimal perturbation'' $\xci[,ij]$ of the intersection loci $A_i^s \cap B_j^s$ away from $\span(A_i^s \cup B_j^s)$.

\begin{introtheorem}\label{introthm:reg}
  With the notation above,
  \[
    \reg_t(A_1,A_2;B_1,B_2) = (A_1^s, A_2^s; B_1^s, B_2^s)_{\xci}.
  \]
\end{introtheorem}

The heavy lifting in this theorem is done by the definitions, but the result is suggestive:

\begin{corCprime}\label{introcor:limit}
  In the context of Corollary~\ref{introcor:asymptotics}, with $R=\co_{\Delta,0}$, we have
  \[
    u(0) = (A_1^s, A_2^s; B_1^s, B_2^s)_{\xci}.
  \]
\end{corCprime}

In other words, the limit period $u(0)$ can now be seen as a degenerate cross-ratio on the \emph{central fibre}. Here, $\xci$ is infinitesimal data on the deformation of the ``singularities'' of the limit cycle; it does not need the full family that leads to the limit period.

The corollary above is in the same vein as \cite[Main~Theorem]{BlochDeJongSertoz2023}, which studies limits of periods of smooth curves, and \cite{Beilinson2025}, which studies the limits of periods in odd dimensional nodal degenerations. Our setup, although certainly special, considers limits of mixed periods, in arbitrary dimensions, and arbitrarily bad degenerations of the cycles, making for an interesting testing ground.

However, in \cite[Main~Theorem]{BlochDeJongSertoz2023} there is more: when the central fibre is defined over a number field, an explicit correction term is given that transforms the limit period to a canonical arithmetic expression on the central fibre. We now pursue an analogous result.

\medskip

Suppose that $\kappa$ is a number field of class number $1$. The integral model of the special fibre singles out an arithmetic perturbation $\xca$ well-defined up to integral units $\co_\kappa^\times$, see Definition~\ref{def:arithmetic_perturbation}. Accordingly, we define the \emph{arithmetic cross-ratio}
\[
  (A_1^s, A_2^s; B_1^s, B_2^s)_\kappa \coleq (A_1^s, A_2^s; B_1^s, B_2^s)_{\xca} \in \kappa^\times/\co_\kappa^\times,
\]
where the notation emphasises that the dependence is only on $\kappa$. The norm of the arithmetic cross-ratio is well defined in $\Q^\times/\{\pm 1\}$ and equals the alternating product of lattices induced by spreading out the configuration over $\co_\kappa$, see Lemma~\ref{lem:norm_of_crossratio}. 

When $\kappa$ does not have class number $1$, it is possible to define the cross-ratio as a fractional ideal, see Remark~\ref{rem:tempt}, but we have not investigated this issue further.

\medskip

Let us return to the set-up of Corollary~\ref{introcor:asymptotics}; however, we now assume that the central fibre can be defined over a number field $\kappa$ of class number $1$. To simplify the set-up, suppose the family of configurations is algebraic: i.e., it is obtained from a family defined over a quasi-projective curve over $\kappa$. 

Now pass to $R = \kappa[[t]]$, a DVR. The algebraicity hypothesis guarantees that for each embedding $\sigma\colon \kappa \toi \C$ the defining equations of the family of planes will be convergent near $t=0$. For each such $\sigma$, the complexified family has a limit period $u_\sigma(0)$ as in Corollary~\ref{introcor:asymptotics}.

We now compare the value of $u_\sigma(0)$ to the arithmetic cross-ratio. This requires comparing the alternating tensors $\xci^\otimes$ and $\xca^\otimes$, the latter of which is defined only up to integral units. Therefore, their respective ratio $\xci^\otimes/\xca^\otimes$ defines an element of $\kappa^\times/\co_\kappa^\times$. 

\begin{corCsecond}\label{introcor:norm}
  The limit periods satisfy
  \[
    \log\abs{u_\sigma(0)}
      = \log\abs{(A_1^s,A_2^s;B_1^s,B_2^s)_\kappa}_\sigma + \log\abs{\xci^\otimes/\xca^\otimes}_\sigma \in \R/\log\abs{\co_\kappa^\times}_\sigma,
  \]
  and summing over all embeddings we get
  \[
    \sum_{\sigma\colon\kappa\hookrightarrow\C} \log\abs{u_\sigma(0)}
      = \log N (A_1^s,A_2^s;B_1^s,B_2^s)_\kappa + \log N(\xci^\otimes/\xca^\otimes) \in \R.
  \]
\end{corCsecond}

The term correcting the arithmetic cross-ratio is the analogue of the Kodaira--Spencer term in~\cite[Main~Theorem]{BlochDeJongSertoz2023} and compares the rate of separation induced by the family against the integral normal directions. 

We emphasize that the norm of the arithmetic cross-ratio can be expressed as indices of $\co_\kappa$-lattices, see~Lemma~\ref{lem:norm_of_crossratio}, which appears above as a limit period. Furthermore, the correction term $\log N(\xci^\otimes/\xca^\otimes)$ depends on the choice of uniformiser; the ``arithmetically meaningful'' uniformiser is arguably the one for which this term vanishes. Such a uniformiser exists, as is clear upon inspection of~\eqref{eq:arch_divergence}, when $\deg\big((\underline A_1-\underline A_2)\cdot(\underline B_1-\underline B_2)\big) \neq 0$.

\subsection{Acknowledgements} 
We thank Spencer Bloch and the organisers of the Regulators V conference in Pisa. We thank Zhelun Chen and Robin de Jong for fruitful discussions. We are grateful to Riccardo Pengo for suggestions on an earlier draft.

This paper is adapted from the first named author's master's thesis~\cite{Bakker2026} from Leiden University and Delft University, which at 92 pages is a leisurely introduction to many of the topics covered here. Theorem~\ref{introthm:aug_height_is_cross} here and Theorem E there agree. Our proof of Theorem~\ref{introthm:dvr} is a streamlined version of the one given there~\cite[Thm~3.19]{Bakker2026}. Theorem~\ref{introthm:reg} here and its corollary are inspired by, though greatly generalise, Theorem D there.

%% file: sections/01_correspondence.tex
\section{The incidence correspondence over the Grassmannian}

Let $K\subset \C$ be a field and let $V$ be a $(n+1)$-dimensional $K$-vector space. We let $\Gr(k+1,V)$ denote the space of $(k+1)$-dimensional linear subspaces of $V$. We also use the notation $\G(k,\pv)$, to refer to $\Gr(k+1,V)$, when working with projective spaces. When $V = K^{n+1}$, we will use $\Gr(k+1, n+1)$ and $\G(k, n)$ instead.

For a $k$-plane $A \subset \pn$ we will write $[A]\in \G(k,\pv)$ for the corresponding point in the Grassmannian. Let $A^\between  \subset \G(n-k-1,\pv)$ denote the locus of $(n-k-1)$-planes intersecting $A$. 

With respect to the Plücker embedding, $\G(n-k-1,\pv) \subset \Pp(\bigwedge^{n-k} V)$, the  subspace
$A^\between$ is cut out by the linear functional $a^\perp$, where $a^\perp$ is a generator of $\det {\hat A}^\perp$.

Now given a complemetary configuration of type $(k,n)$, $A_i,\, B_j\subset \pv$, we consider the cohomology group 
\begin{equation}\label{eq:hg}
\ch_\G \coleq \H^{1}(\G(k,n) - B_1^\between \cup B_2^\between, \{[A_1],[A_2]\})(1)
\end{equation}
with its mixed Hodge structure. Since $\G(k,n)$ is simply connected, it follows that
\[\ch_\G \in \Ext^1_{\ZMHS}(\Z(0),\Z(1)),\]
where we choose the generators $A_1-A_2$ and $B_1^\between - B_2^\between$ for the two extremal graded pieces of $\ch_\G$.

\begin{proposition}\label{prop:grass_height}
With the usual identification of the extension group $\Ext^1_{\ZMHS}(\Z(0),\Z(1))$ with $\C^\times$, the value corresponding to $\ch_\G$ equals the cross-ratio $(A_1,A_2;B_1,B_2)$.
\end{proposition}
\begin{proof}
    
    Let $f$ be the restriction of the rational function $b_1^\perp / b_2^\perp$ to $\G(k,\pv)$. The divisor $B_1^\between-B_2^\between$ is the principal divisor associated to $f$. Therefore, the form $\dd \log(f) = \frac{\dd f}{f}$ lifts the generator of $\Z(0)$ to $F^0\ch_\G$. 

    Pick a smooth $1$-chain $\gamma$ with boundary $\partial \gamma = A_1 - A_2$. The value of the integral modulo $\Z \tpip$ of $\dd \log (f)$ over $\gamma$ is independent of the choice of this path, giving
    \[\int_\gamma \dd\log(f) = \log(f)(A_1) - \log(f)(A_2) \mod \Z\tpip .\]

    The exponential of this integral is the extension class for $\ch_\G$, which is a general fact for $\H_1 \simeq \H^1(1)$ when there is a divisorial boundary oriented by a principle divisor relative to the boundary of a path. Clearly, this exponential evaluates to the cross-ratio formula.
\end{proof}

\begin{theorem}\label{thm:compare_extensions}
The $\ch_\Pp$ from~\eqref{eq:hp} and $\ch_\G$ from~\eqref{eq:hg} are isomorphic as extensions of mixed Hodge structures.
\end{theorem}

\begin{proof}
We have the incidence correspondence
\[
\begin{tikzcd}[row sep=large]
& \ci = \{(\Pi, x) \in \G(k,\pv) \times \pv \mid x \in \Pi\}
\arrow[dl,"p"'] \arrow[dr,"q"] & \\
\G(k,\pv) & & \pv
\end{tikzcd}
\]
where $p$ and $q$ are the projections to $\G(k,\pv)$ and $\pv$, respectively.

For notational convenience let us write: $\G$ and $\Pp$ for the Grassmannian and the projective space, $A = A_1 \cup A_2, B= B_1 \cup B_2, [A]=\{[A_1],[A_2]\}$, and $B^\between = B^\between_1 \cup B^\between_2$. 

The morphisms $p,\, q$ of the correspondence induce morphisms 
\[\begin{tikzcd}
	{(\G - B^\between,[A])}  & {(\mathcal{I} - q^{-1}B,p^{-1}[A])} &  {(\Pp - B,A)}.
	\arrow["p"', from=1-2, to=1-1]
	\arrow["q", from=1-2, to=1-3]
\end{tikzcd}\]

    In particular, we get morphisms of MHS 
    \[
    \H^{2k+1}(\Pp - B, A )(k+1) \overset{q^*}{\too} \H^{2k+1}(\ci - q^{-1}(B),p^{-1}(A))(k+1) \overset{p_!}{\too} \H^1(\G - B^\between, [A])(1),
    \]
    where the latter morphism is ``integration along fibres'', the Lefschetz dual to $p^*$, see \cite[p.8]{Steenbrink1999} or \cite[Theorem 4.1]{deCataldoMigliorini2004}.

    The composition gives a morphism $\phi \colon \ch_\Pp \to \ch_\G$. To see that this is an isomorphism as extensions, we need only show that the generators of the weight graded pieces map to one another.

    Note, for instance, that the generator of $\Z(1)$ in $\ch_\Pp$ comes from the coboundary map $\H^{2k}(A)(k+1) \to \H^{2k+1}(\Pp - B, A )(k+1)$. Looking at what $p$ and $q$ induce along the coboundary component for all three cohomologies we get
    \[
    \H^{2k}(A)(k+1) \overset{q^*}{\too} \H^{2k}(p^{-1}(A))(k+1) \overset{p_!}{\too} \H^0([A])
    \]
    both of which are evidently isomorphisms. 
    
    To prove that the generators of $\Z(0)$ map to one another, we argue analogously this time using the residue maps instead of the coboundary maps.
\end{proof}

\begin{proof}[Proof of Theorem~\ref{introthm:aug_height_is_cross}]
    The augmented height pairing is the extension class of $\ch_\Pp$, which we compute by combining Theorem~\ref{thm:compare_extensions} with Proposition~\ref{prop:grass_height}.
\end{proof}

%% file: sections/02_intermezzo.tex
\section{Intermezzo: properties of the cross-ratio}\label{sec:intermezzo}

In this section we record equivalent formulations of the cross-ratio for a complementary configuration $A_i,\,B_j$ and the action of the admissible permutations on the planes. In particular, when all four planes have the same dimension, the full $S_4$-action can be considered and the change in the values of the cross-ratio can be directly compared. We also show that the cross-ratio is unchanged when projecting from $A_1 \cap A_2$ (resp.\ $B_1\cap B_2$) or when restricting to the subspaces $\span(A_1 \cup A_2)$ (resp.\ $\span(B_1 \cup B_2)$). 

\subsection{Equivalent formulations of the cross-ratio}
Let $A_i,\,B_j \subset \pv$ be a complementary configuration of type $(k,n)$. As in the introduction, choose $a_i$'s, $a_i^\perp$'s, $b_j$'s and $b_j^\perp$'s.
 Via the inclusions $\hat{A}_i, \hat{B}_j\xhookrightarrow{} V$ and $\hat{A}_i^\perp, \hat{B}_j^\perp\xhookrightarrow{} V^\vee$ we obtain elements $a_i\wedge b_j\in \det V$ and $a_i^\perp\wedge b_j^\perp\in \det V^\vee$.
 Since $\det V$ is one dimensional, any two non-zero vectors in it have a well defined ratio in $K^\times$, similarly for $\det V^\vee$.
 
 We now observe the following equalities:
    \[
    (A_1,A_2;B_1,B_2) \coleq \frac{b_1^\perp(a_1)}{b_1^\perp(a_2)}\frac{b_2^\perp(a_2)}{b_2^\perp(a_1)}= \frac{a_1\wedge b_1}{a_1\wedge b_2}\frac{a_2\wedge b_2}{a_2\wedge b_1} = \frac{a^\perp_1\wedge b^\perp_1}{a^\perp_1\wedge b^\perp_2}\frac{a^\perp_2\wedge b^\perp_2}{a^\perp_2\wedge b^\perp_1}
    \]

    Choose an isomorphism $\varphi \colon \det V \simeq K$ and consider the linear functional $F_i \in \det \hat{B}_i^\perp$ given by $F_i(-) = \varphi\left(-\wedge b_i\right)$. It is clear that $F_i =\lambda_i b_i^\perp$ for some $\lambda_i \in K^\times$, so the scalars cancel in each ratio and
    $$
    \frac{b_1^\perp(a_1)}{b_1^\perp(a_2)}\frac{b_2^\perp(a_2)}{b_2^\perp(a_1)} = \frac{F_1(a_1)}{F_1(a_2)}\frac{F_2(a_2)}{F_2(a_1)} = \frac{a_1\wedge b_1}{a_1\wedge b_2}\frac{a_2\wedge b_2}{a_2\wedge b_1}.
    $$
    The third equality follows by a similar argument. 
    \subsection{Symmetries of the cross-ratio}
    It is clear from the identities above that the generalised cross-ratio has two symmetries:
\begin{equation*}
(A_1,A_2;B_1,B_2) = (B_1,B_2;A_1,A_2), \quad (A_1,A_2;B_1,B_2) = (A_1,A_2;B_2,B_1)^{-1}.
\end{equation*}

\medskip

 If $n=2k+1$ and $A_i,\,B_j$ are pairwise disjoint, then the full permutation group $S_4$ on four letters can act on the input of the cross-ratio. 

To compare the value of the cross-ratios under the permutation action, we introduce some notation.

Decompose $V = \hat{A}_1\oplus \hat{A}_2$. Since the $\hat{B}_j$'s are complementary to both $\hat{A}_i$'s they can be realised as 
\[
\hat{B}_j  = \{v + \varphi_j (v)\mid v\in \hat{A}_1\}
\]
where $\varphi_j \colon \hat{A}_1 \to \hat{A}_2$ is a linear map for $j=1,\,2$.

Then, setting $e \coleq a_1 \wedge a_2$ we have 
\[
    a_1\wedge b_1 = (\det\varphi_1) e, \; 
    a_1\wedge b_2 = (\det\varphi_2)e, \;
    b_1\wedge b_2 = \det(\varphi_2 -\varphi_1)e, \;
    a_2 \wedge b_1 = a_2\wedge b_2 = (-1)^{k+1}e.
\]

Thus we get
\[
(A_1,A_2;B_1,B_2) = \frac{\det\varphi_1}{\det\varphi_2} = \det(\varphi_2^{-1}\varphi_1),
\]
so the cross-ratio is determined by the operator $\varphi_2^{-1}\varphi_1 \colon \hat{A}_1\to \hat{A}_1$. Let $\lambda_0, \dots, \lambda_k$ be the eigenvalues of $\varphi_2^{-1}\varphi_1$. Now the proof of the proposition below is immediate.

\begin{proposition}\label{prop:symmetries}
  The cross-ratio satisfies the following symmetries tabulated below:
  \medskip
    \begin{center}
    \small
    \setlength{\tabcolsep}{3pt}
    \begin{tabular}{@{}cccccc@{}}
      $(A_1,A_2;B_1,B_2)$ & $(A_1,A_2;B_2,B_1)$ & $(A_1,B_1;A_2,B_2)$ & $(A_1,B_1;B_2,A_2)$ & $(A_1,B_2;A_2,B_1)$ & $(A_1,B_2;B_1,A_2)$ \\
      \midrule
      $\displaystyle\prod_{i=0}^k \lambda_i$ & $\displaystyle\prod_{i=0}^k \lambda_i^{-1}$ & $\displaystyle\prod_{i=0}^k (1-\lambda_i)$ & $\displaystyle\prod_{i=0}^k (1-\lambda_i)^{-1}$ & $\displaystyle\prod_{i=0}^k \frac{\lambda_i-1}{\lambda_i}$ & $\displaystyle\prod_{i=0}^k \frac{\lambda_i}{\lambda_i-1}$ \\
    \end{tabular}
    \end{center}
\end{proposition}

\begin{remark}
The formulas of \cite[Theorem 3]{HollandSparling2025} agree with Proposition \ref{prop:symmetries}, however our proof does not require the assumption that $\Char K\neq 2$ imposed there. 
\end{remark}

\subsection{Passing to quotients and subspaces}
Let $A_i,\,B_j\subset \pv$ be a complementary configuration of type $(k,n)$. Let $S=\hat{A}_1 \cap \hat{A}_2$, of dimension $\dim S\ge 0$.

The cross-ratio remains unchanged when passing to the quotient space $V/S$. That is, the spaces $A'_i = \mathbb{P}( \hat{A}_i / S)$ and $B'_j = \mathbb{P}((\hat{B}_j + S)/S)$ give a complementary configuration of type $(k-\dim S,\,n-\dim S)$ in $\Pp(V/S)$ and
\[
(A_1,A_2;B_1,B_2) = (A'_1,A'_2;B'_1,B'_2).
\]

To prove the equality above, note that the generator $ s \in\det S$ divides the generators $a_1$ and $a_2$, i.e., $a_i = s \wedge a_i'$. Since all four terms $a_i\wedge b_j$ pick up this same factor, removing $s$ preserves the ratio. 

Dually, suppose $\hat{A}_1+ \hat{A}_2\subset W$ where $W\subset V$. In this case, restricting $b^\perp_j$'s to $\bigwedge^{k+1}W$ does not change the pairing with the $a_i$'s. So writing $A'_i$ for the subspace $A_i$, now viewed as a linear subspace of $\mathbb{P}(W)$, and $B'_j = B_j\cap \mathbb{P}(W)$ we conclude
\[
(A_1,A_2;B_1,B_2) = (A'_1,A'_2;B'_1,B'_2).
\]

\begin{remark}
    Goncharov~\cite{Goncharov2026} uses the second reduction on the cross-ratio side of his inductive proof.
\end{remark}

%% file: sections/03_intersection.tex
\section{Intersection degree and the valuation of the cross-ratio}\label{sec:intersect}

In this section we express the intersection degree of two complementary linear subspaces of a projective space over a DVR as the valuation of a determinant, and thereby prove Theorem~\ref{introthm:dvr}. The argument combines lattice linear algebra with Fulton's intersection theory.

\medskip

Let $R$ be a DVR, with quotient field $K$, residue field $\kappa$, and valuation $v_R$. Let $t$ be a uniformiser of $R$.

Let $\cv \simeq R^{n+1}$ be a free $R$-module of rank $n+1$ and let $V \coleq \cv \otimes_R K \simeq K^{n+1}$. We view $\cv \toi V$ as an $R$-submodule, and throughout write a calligraphic letter for an $R$-lattice and the matching latin letter for its base change to $K$.

Let $\cv^\vee$ denote the $R$-linear dual of $\cv$, $\hom_R(\cv,R)$. We define $\Pp(\cv) \coleq \proj(\Sym^* \cv^\vee) \simeq \pn_R$. Note that the generic fibre of $\Pp(\cv)$ over $\Spec(R)$ is $\Pp(V) \simeq \pn_K$.

For a $K$-subspace $W \subset V$, let $\cw \coleq W \cap \cv$. Since $\cv$ is free and $R$ is a DVR, $\cw$ is free of rank $\dim W$ and saturated, i.e.\ $\cv/\cw$ is torsion free. The surjection $\cv^\vee \tos \cw^\vee$ then induces a closed immersion $\Pp(\cw) \toi \Pp(\cv)$~\cite[Tag~01N1]{stacks-project}, realising $\Pp(\cw)$ as the Zariski closure of $\Pp(W) \subset \Pp(V)$.

Suppose now that $\cw_1, \cw_2 \subset \cv$ are both saturated and $V = W_1 \oplus W_2$; this means that the map $\mu \colon \cw_1 \to \cv/\cw_2$ has a cokernel that is pure torsion.

If we pick generators $w_i \in \det(\cw_i)$ then observe that $w_i$ will also serve as a generator for $\det(W_i)$. Furthermore, use the map $\cw_1 \oplus \cw_2 \to \cv$ to view $w_1 \wedge w_2$ as an element of $\det(\cv)$.

\begin{lemma}\label{lem:wedge-det}
  Let $w_i$ generate $\det \cw_i$ and regard $w_1\wedge w_2$ as an element of $\det \cv$ via $\cw_1\oplus\cw_2 \to \cv$. Then
  \[
    v_R\big(\det(\cw_1\oplus\cw_2 \to \cv)\big) = v_R(\det \mu) = v_R(w_1\wedge w_2).
  \]
\end{lemma}
\begin{proof}
  Saturation of $\cw_2$ gives a splitting $\cv \simeq \cv/\cw_2 \oplus \cw_2$; in bases adapted to it the map $\cw_1\oplus\cw_2 \to \cv$ is block lower-triangular with diagonal blocks $\mu$ and $\id_{\cw_2}$, hence has determinant $\det\mu$. Fixing a generator $e$ of $\det \cv$, this determinant is by definition the scalar with $w_1\wedge w_2 = \det(\cw_1\oplus\cw_2 \to \cv)\cdot e$.
\end{proof}

The proof below uses the intersection theory of \cite[Ch.~14]{Fulton1998}. Although that chapter is written over a field, the formulas we use are valid over any regular base, in particular over our DVR~\cite[p.~395]{Fulton1998}. Recall that $\deg$ denotes the degree of the zero-cycles, which in this case are all supported over the closed point of $\Spec R$.

\begin{lemma}
  \label{lem:intersection}
  Let $\cw_1,\cw_2\subseteq\cv$ be saturated with $V = W_1\oplus W_2$. Then
  \[
   \deg ([\Pp(\cw_1)]\cdot_{\Pp(\cv)} [\Pp(\cw_2)]) = v_R(\det \left( \cw_1\oplus \cw_2 \to \cv \right) ).
  \]
\end{lemma}
\begin{proof}
  Let $\co_{\Pp(\cv)}(-1) \to \cv \otimes \co_{\Pp(\cv)}$ be the universal line bundle over $\Pp(\cv)$. Following through with the map $\cv \tos \cv/\cw_2$ and then twisting leads to a section
  \[
    \sigma \colon \co_{\Pp(\cv)} \to \cv/\cw_2 \otimes \co_{\Pp(\cv)}(1).
  \]
  Observe that the zero locus $Z(\sigma)$ of $\sigma$ is $\Pp(\cw_2)$, and as it is a regular section, the class $[\Pp(\cw_2)]$ coincides with the top Chern class of $\sigma$~\cite[\S14.1]{Fulton1998}. 

  Since $i \colon \Pp(\cw_1) \toi \Pp(\cv)$ is a local complete intersection, the desired intersection class can be computed as the degeneracy locus of the pullback section $i^*\sigma$~\cite[Prop. 14.1.d.ii]{Fulton1998},
  \[
    [\Pp(\cw_1)]\cdot [\Pp(\cw_2)] = i^*\Zloc(\sigma) = \Zloc(i^*\sigma)
  \]
  where $\Zloc(-) \in A_\bullet(Z(-))$ is the localized top Chern class of \emph{loc.\ cit.}\ (our $\Zloc(-)$ is written $\mathbb{Z}(-)$ there).

  Let $s \coleq i^*\sigma \colon \co_{\Pp(\cw_1)} \to E \coleq \cv/\cw_2 \otimes \co_{\Pp(\cw_1)}(1)$. Notice that $Z(s) \subset \Pp(\cw_1) \cap \Pp(\cw_2)$, and this intersection lies entirely in the central fibre $D \coleq \Pp(\cw_1 \otimes_R \kappa) \simeq \Pp^k_\kappa$, where $k+1 = \rank \cw_1$, of $\Pp(\cw_1)$.

The restriction of the universal line bundle from $\Pp(\cv)$ to $\Pp(\cw_1)$ factors through the universal line bundle on the latter, so that we may write $s$ as the composition
\[
  \co_{\Pp(\cw_1)} \to \cw_1 \otimes \co_{\Pp(\cw_1)}(1) \overset{\mu \otimes \id}{\too} \cv/\cw_2 \otimes \co_{\Pp(\cw_1)}(1),
\]
where the first is the universal line bundle. Note that the first map is injective everywhere and the latter is an isomorphism on $\Pp(W_1)$.

Use the Smith normal form (SNF) on $\mu$ to split the vector bundles as follows
\[
  \mu\otimes \id \colon \co_{\Pp(\cw_1)}(1)^{k+1} \xrightarrow{\diag(t^{\alpha_0}, \dots, t^{\alpha_k})} \co_{\Pp(\cw_1)}(1)^{k+1}  
\]
with $0 \le \alpha_0 \le \dots \le \alpha_k$.

If $\alpha_0 > 0$, that is $t \mid \mu$, then $Z(s)$ is the whole central fibre $D$, a divisor in $\Pp(\cw_1)$, and the residual formula~\cite[Example 14.1.4]{Fulton1998} gives
\[
  \Zloc(s) = \Zloc(s/t) + c_{k}(E)\cap D,
\]
where we used that $D \cdot D = 0$. Let $h = c_1(\co(1))$ be the hyperplane class; then $c_k(E) = c_k(\co(1)^{k+1}) = (k+1) h^k$, and $h^k \cap D = [\mathrm{pt}]$ is the class of a point in $A_0(D)$. 

The codomain of $s/t$ is $E \otimes \co(-D) \simeq E$ and therefore we can repeat this argument to find
\[
  \Zloc(s) = \Zloc(s/t^{\alpha_0}) + (k+1) \alpha_0 [\mathrm{pt}].
\]
When $\alpha_0 = 0$ this identity holds trivially and we may continue to the next step.

Let $s' = s/t^{\alpha_0}$ and write $E = L' \oplus E'$ where $L' = \co(1)$ is the first component in the SNF and $E'$ is the rest. Correspondingly, we decompose $s' = s_1 \oplus s_2$ with $s_1$ landing in $L'$ and $s_2$ in $E'$. Applying \cite[Example 14.1.3]{Fulton1998} we find
\[
  \Zloc(s') = h \cap \Zloc(s_2).
\]

Now we can repeat the process with $s_2$. If $\alpha_1-\alpha_0 > 0$ then $\Zloc(s_2)$ is supported on the divisorial locus $D$. This time we use $h \cap c_{k-1}(E') = k [\mathrm{pt}]$ to obtain
\[
  h \cap \Zloc(s_2) = h\cap \Zloc(s_2/t^{\alpha_1-\alpha_0}) + k(\alpha_1-\alpha_0)[\mathrm{pt}].
\]

We continue to decompose and factor out powers of $t$; the $j$-th step contributes $(k+1-j)(\alpha_j - \alpha_{j-1})[\mathrm{pt}]$, with $\alpha_{-1} \coleq 0$. The sum telescopes, giving
\[
  \Zloc(s) = ( \alpha_0 + \alpha_1 + \dots + \alpha_k )[\mathrm{pt}].
\]
Its degree is $v_R(\det \mu)$, which combined with Lemma~\ref{lem:wedge-det} concludes the proof.
\end{proof}

\begin{remark}
  After completing this work we learned that Theorem~\ref{introthm:dvr} also follows from \cite[Thm.~3.4]{Werner2001}, whose proof is of a rather different nature. Although \emph{loc.\ cit.}\ assumes $K$ to be a finite extension of $\Q_p$, it appears that their proof goes through without this hypothesis.
\end{remark}

\begin{proof}[Proof of Theorem~\ref{introthm:dvr}]
  Write $\hat\ca_i \coleq \hat A_i \cap \cv$ and $\hat\cb_j \coleq \hat B_j \cap \cv$; these are saturated in $\cv$, and $\underline{A}_i = \Pp(\hat\ca_i)$, $\underline{B}_j = \Pp(\hat\cb_j)$ are the Zariski closures of $A_i, B_j$ in $\pn_R$. Recall from Section~\ref{sec:intermezzo} that the cross-ratio is an alternating product of the wedges $a_i \wedge b_j \in \det V$; scale the generators so that $a_i$ generates $\det \hat\ca_i$ and $b_j$ generates $\det \hat\cb_j$.

  By Lemma~\ref{lem:wedge-det}, the valuation of $a_i \wedge b_j \in \det \cv$ equals $v_R(\det(\hat\ca_i \oplus \hat\cb_j \to \cv))$, which by Lemma~\ref{lem:intersection} is the intersection degree $\deg\!\big([\underline A_i] \cdot [\underline B_j]\big)$. Bilinearity of the intersection product now gives the claim.
\end{proof}

%% file: sections/04_limit.tex
\section{Perturbations of the cross-ratio}\label{sec:limit}

Throughout this section we use notation local to this section and to the linear algebra at hand, free of the decorations used in the introduction.

\subsection{The determinantal identity}

We will prove the classical expression~\eqref{eq:detV} for the reader's convenience and describe the isomorphism on individual elements for later use.

Let $K$ be a field and $V$ be a finite-dimensional $K$-vector space. Let $A, \, B \subset V$ be subspaces of complementary dimension, $\dim A + \dim B = \dim V$. Let $S = A \cap B$ and $Q = V/(A+B)$. 

Note the two short exact sequences
\[
  0 \to S \to A \oplus B \to (A+B) \to 0, \qquad 0 \to (A+B) \to V \to Q \to 0.
\]
Take determinants
\[
  \det A \otimes \det B \simeq  \det S \otimes \det(A+B), \qquad \det V \simeq  \det (A+B) \otimes \det Q.
\]
Eliminating $\det(A+B)$ we find
\begin{equation}
  \label{eq:detV_in_proof}
  \det V \simeq \det A \otimes \det B \otimes (\det S)^\vee \otimes \det Q \simeq \det A \otimes \det B \otimes \hom(\det S, \det Q).
\end{equation}

To be more precise, let $a \in \det A,\, b \in \det B, \, \phi \in \hom(\det S, \det Q)$ be generators. We will describe the image of $a \otimes b \otimes \phi$ in $\det V$. Pick a generator $s \in \det S$ and split $B = B' \oplus S$, then there is a unique generator $b' \in \det B'$ such that $b = b' \wedge s$. Then
\begin{equation}
  a \otimes b \otimes \phi = a \otimes (b' \wedge s) \otimes \phi \mapsto a \wedge b' \wedge \phi(s) \in \det V.
\end{equation}

\subsection{The limit perturbation}

Let $R$ be a DVR with fraction field $K$, residue field $\kappa$, valuation $v_R$, and uniformiser $t$. Fix a free $R$-module $\cv \simeq R^{n+1}$ and set $V \coleq \cv \otimes_R K$ and $V^s \coleq \cv \otimes_R \kappa$; throughout, a calligraphic letter denotes an $R$-lattice and the matching latin letter its base change to $K$. Each $K$-subspace $W \subset V$ spreads out to the saturated lattice $\cw \coleq W \cap \cv$, and $\Pp(\cw) \subset \pn_R$ is the Zariski closure of $\Pp(W) \subset \pn_K$.

Take a decomposition $V = W_1 \oplus W_2$, so that $\Pp(W_1)$ and $\Pp(W_2)$ are disjoint planes in $\pn_K$, and let $\Psi \colon \cw_1 \oplus \cw_2 \to \cv$ be the inclusion of their spread-out lattices. Since $\Psi$ recovers the isomorphism $W_1 \oplus W_2 \simeq V$ on the generic fibre, its cokernel is pure torsion. Write $W_i^s = \cw_i \otimes_R \kappa$, and set $S = W_1^s \cap W_2^s$ and $Q = V^s/(W_1^s + W_2^s)$.

\begin{proposition}\label{prop:limit-perturbation}
  The inclusion $\Psi$ induces an isomorphism $\det W_1^s \otimes \det W_2^s \simeq \det V^s$ and, therefore, an element $\xci \in \hom(\det S, \det Q)$ (the ``limit perturbation''), which is canonical up to the choice of $t \bmod \xm^2$, the uniformiser to first order.
\end{proposition}

\begin{proof}
  The map $\det \Psi \colon \det \cw_1 \otimes \det \cw_2 \to \det \cv$ is a scalar of valuation $v_R(\det \Psi)$; or, more naturally, this valuation is the length of the cokernel of $\det \Psi$. 

 We define the regularisation of $\det \Psi$ as the map
  \[
    \reg_t \det \Psi \coleq \bigl(t^{-v_R(\det \Psi)} \det \Psi \bmod \xm\bigr) \colon \det W_1^s \otimes \det W_2^s \simeq \det V^s,
  \]
 which is an isomorphism, as the scaled map has cokernel of length $0$. Rewriting the determinantal identity~\eqref{eq:detV_in_proof} for the special fibre gives a canonical isomorphism
  \[
    \hom(\det W_1^s \otimes \det W_2^s, \det V^s) \simeq \hom(\det S, \det Q),
  \]
  under which $\reg_t \det \Psi$ corresponds to an element of $\hom(\det S, \det Q)$, the limit perturbation $\xci$ for this pair. Since $\reg_t$ depends on $t$ only through $t \bmod \xm^2$, so does $\xci$.
\end{proof}

\begin{remark}
  In fact, $\det \Psi$ induces a canonical isomorphism $\det W_1^s \otimes \det W_2^s \simeq \det V^s \otimes (\xm/\xm^2)^{\otimes v}$, where $v = v_R(\det \Psi)$, which is independent of the choice of a uniformiser.
\end{remark}

\begin{proof}[Proof of Theorem~\ref{introthm:reg}]
  Scale the generators so that $a_i$ generates $\det(\hat A_i \cap \cv)$ and $b_j$ generates
  $\det(\hat B_j \cap \cv)$, with reductions $a_i^s, b_j^s$ modulo $\xm$. For each pair,
  Proposition~\ref{prop:limit-perturbation} applied to $W_1 = \hat A_i$, $W_2 = \hat B_j$ produces
  $\xci[,ij]$. The construction of $\xci[,ij]$ gives the first equality below and Lemma~\ref{lem:wedge-det} the second:
  \[
    a_i^s \otimes b_j^s \otimes \xci[,ij] = (\reg_t \det \Psi)(a_i^s \otimes b_j^s) = \reg_t(a_i \wedge b_j) \ \in \det V^s,
  \]
  where we made the canonical identification~\eqref{eq:detV_in_proof}.

  As $\reg_t$ is multiplicative and the cross-ratio is the alternating product of the wedges $a_i \wedge b_j$, see Section~\ref{sec:intermezzo}, the alternating product of the left-hand sides is $(A_1^s, A_2^s; B_1^s, B_2^s)_{\xci}$ and of the right-hand sides is $\reg_t(A_1, A_2; B_1, B_2)$.
\end{proof}

\subsection{The arithmetic perturbation}

Let $K$ be a number field and $R = \co_K$ its ring of integers. Let $\cv$ be a free $R$-module of finite rank and $V = \cv \otimes_R K$. Let $A, B \subset V$ be subspaces of complementary dimension, with $S = A \cap B$ and $Q = V/(A+B)$. Spread out to saturated lattices $\cs, \ca, \cb \subset \cv$, with $\cs = S \cap \cv$ etc., and $\cq = \cv/(\ca + \cb)^{\sat}$.

Over a Dedekind domain a saturated submodule of a free module is projective but need not be free, so the determinant lines $\det \ca, \det \cb, \det \cs, \det \cq$ are invertible $R$-modules, possibly nontrivial in the class group. In general, they have no canonical generator. The following proposition suggests that perhaps we should, in general, take the arithmetic cross-ratio to be an alternating product of fractional ideals. However, we will later specialise to PIDs in order to get back numerical values.

Let $\cl \coleq \ca + \cb$ and $\cl' \coleq (\ca + \cb)^{\sat}$ and $\tau \coleq \cl'/\cl$, which is a finite-length $R$ module. We have $ \det \cl \subset \det \cl'$ and thus $\det \cl \otimes (\det \cl')^{-1} \subset R$, and the latter is the \emph{order ideal} $\chi(\tau)$ of $\tau$.

\begin{proposition}\label{prop:arithmetic_perturbation}
  Under the identification~\eqref{eq:detV_in_proof}, and the inclusions $\det \ca \subset \det A$ etc., the image of the natural map 
  \[
    \det \ca \otimes \det \cb \otimes \hom(\det \cs, \det \cq) \to \det \cv
  \]
  equals $\chi(\tau) \cdot \det \cv$.
\end{proposition}

\begin{proof}
  The short exact sequences
  \[
    0 \to \cs \to \ca \oplus \cb \to \cl \to 0, \quad
    0 \to \cl \to \cl' \to \tau \to 0, \quad
    0 \to \cl' \to \cv \to \cq \to 0
  \]
  give, upon taking determinants, the canonical isomorphisms
  \[
    \det \ca \otimes \det \cb \simeq \det \cs \otimes \det \cl, \quad
    \det \cl = \chi(\tau) \cdot \det \cl', \quad
    \det \cv \simeq \det \cl' \otimes \det \cq.
  \]
   Tensoring the first by $\hom(\det \cs, \det \cq)$, so that $\det \cs \otimes \hom(\det \cs, \det \cq) = \det \cq$, and combining,
  \[
    \det \ca \otimes \det \cb \otimes \hom(\det \cs, \det \cq) \simeq \det \cq \otimes \det \cl = \chi(\tau) \cdot \bigl(\det \cq \otimes \det \cl'\bigr) = \chi(\tau) \cdot \det \cv,
  \]
  as lattices in $\det V$.
\end{proof}

\begin{remark}\label{rem:tempt}
  For a complementary configuration of planes, it is tempting to define a cross-ratio as a fractional ideal, the alternating product of the $\chi(\tau)$'s.
\end{remark}

\medskip

\emph{Suppose for the rest of this section that $K$ has class number $1$}, i.e., $R$ is a principal ideal domain. Then every one of our determinant lines, in particular $\hom(\det \cs, \det \cq)$, is free. 

  \begin{definition}\label{def:arithmetic_perturbation}
  Take the \emph{arithmetic perturbation} $\xca$ to be a generator of $\hom(\det \cs, \det \cq)$, which is well-defined up to $R^\times$. 
\end{definition}

\begin{definition}\label{def:arithmetic_crossratio}
  For a degenerate configuration $A_i, B_j$ we pick arithmetic perturbations $\xca[,ij]$ for each pair $A_i,\, B_j$. And then define the \emph{arithmetic cross-ratio} to be
  \[
    (A_1, A_2; B_1, B_2)_K = \frac{a_1 \otimes b_1 \otimes \xca[,11]}{a_1 \otimes b_2 \otimes \xca[,12]} \, \frac{a_2 \otimes b_2 \otimes \xca[,22]}{a_2 \otimes b_1 \otimes \xca[,21]} \in K^\times/R^\times.
  \]
\end{definition}

Spreading out $A_i$ to $\ca_i$ etc., define the pure torsion $R$-module $\tau_{ij} \coleq (\ca_i + \cb_j)^{\sat}/(\ca_i + \cb_j)$.

\begin{lemma}\label{lem:norm_of_crossratio}
  The norm of the arithmetic cross-ratio is well-defined and equals
  \[
    \bigl\lvert N(A_1, A_2; B_1, B_2)_K \bigr\rvert = \frac{\#\tau_{11} \, \#\tau_{22}}{\#\tau_{12} \, \#\tau_{21}} \in \Q_{>0}.
  \]
\end{lemma}
\begin{proof}
  Choosing generators $a \in \det \ca$ and $b \in \det \cb$, the element $a \otimes b \otimes \xca$ can be identified, via~\eqref{eq:detV_in_proof} and Proposition~\ref{prop:arithmetic_perturbation}, with a generator of the principal ideal $\chi(\tau)$. The norm of this generator will equal $\#\tau$. When computing the cross-ratio, the individual choices for $a$'s and $b$'s do not influence the result, and therefore we can always assume they are integral generators.
\end{proof}